\documentclass[11pt]{amsart}
\usepackage{amsmath,amssymb,bm, color}

\begin{document}
\newtheorem{theorem}{Theorem}[section]
\newtheorem{lemma}[theorem]{Lemma}
\newtheorem{corollary}[theorem]{Corollary}
\newtheorem{prop}[theorem]{Proposition}
\newtheorem{definition}[theorem]{Definition}
\newtheorem{remark}[theorem]{Remark}

%submit http://www.ams.org/editflow/ef/status.php?p_id=82696&cr=*96413A3FEACE2B77930

 \def\ad#1{\begin{aligned}#1\end{aligned}}  \def\b#1{{\bf #1}} \def\hb#1{\hat{\bf #1}}
\def\a#1{\begin{align*}#1\end{align*}} \def\an#1{\begin{align}#1\end{align}}
\def\e#1{\begin{equation}#1\end{equation}} \def\t#1{\hbox{\rm{#1}}}
\def\dt#1{\left|\begin{matrix}#1\end{matrix}\right|}
\def\p#1{\begin{pmatrix}#1\end{pmatrix}} \def\c{\operatorname{curl}}
 \numberwithin{equation}{section} \def\P{\Pi_h^\nabla}

\def\bg#1{{\pmb #1}} 

\title  [Virtual elements]
   {A family of stabilizer-free virtual elements on triangular meshes}

\author {Xuejun Xu}
\address{School of Mathematical Science, \ Tongji University, \ Shanghai, \ 200092, \ China}
	\address{Institute of Computational Mathematics, AMSS, Chinese Academy of Sciences, Beijing, 100190, China}
\email{xxj@lsec.cc.ac.cn}

\author { Shangyou Zhang }
\address{Department of Mathematical
            Sciences, University
     of Delaware, Newark, DE 19716, USA. }
\email{szhang@udel.edu }

\date{}

\begin{abstract} \ \ 
A family of stabilizer-free $P_k$ virtual elements  are constructed on triangular meshes.
When choosing an accurate and proper interpolation, the stabilizer of the virtual 
  elements can be dropped while the quasi-optimality is kept. 
The interpolating space here is the space of continuous $P_k$ polynomials on the
   Hsieh-Clough-Tocher macro-triangle, where the macro-triangle is defined by connecting
  three vertices of a triangle with its barycenter.
We show that such an interpolation preserves $P_k$ polynomials locally and enforces the
  coercivity of the resulting bilinear form.
Consequently the stabilizer-free virtual element solutions converge at the optimal order.
Numerical tests are provided to confirm the theory and to be compared with existing 
   virtual elements.
\vskip .3cm

%\begin{keywords}

% \end{keywords}

\end{abstract}

\maketitle \baselineskip=16pt

\section{Introduction}
In this work, we construct a family of stabilizer-free $P_k$ virtual elements
(\cite{Beirao,Beirao16, Cao-Chen, Cao-Chen2,  Chen1, Chen2, Chen3,
  Feng-Huang,Feng-Huang2, Huang1, Huang2, Huang3})    on triangular meshes.   

For solving 
 the following model equation, 
\an{ \label{p-e} \ad{ -\Delta  u & = f  && \t{in } \ \Omega, \\
            u&=0 && \t{on } \ \partial \Omega, } }
where $\Omega\subset \mathbb{R}^2$ is a bounded polygonal domain and $f\in L^2(\Omega)$,
the weak form reads:  Find $u\in H^1_0(\Omega)$ such that
\an{ \label{w-e} \ad{ (\nabla u, \nabla v) & = ( f,v)  && \forall v \in  H^1_0(\Omega),} }
where $(\cdot, \cdot)$ denotes the $L^2$ inner product on $\Omega$ 
    and we have $|v|_1^2=(\nabla v,\nabla v)$.

Let $\mathcal{T}_h=\{ K \}$ be a quasi-uniform triangular mesh on $\Omega$ with $h$ as the
  maximum size of triangles $K$.  Let $\mathcal{E}_h$ be the set of edges $e$ in $\mathcal{T}_h$.
For $k\ge 1$, the virtual element space is defined as
\an{ \label{t-V-h} \tilde 
    V_h=\{ v\in H^1_0(\Omega) : \tilde v|_{\partial K}\in \mathbb{B}_k(\partial K),
   \Delta \tilde v|_K \in P_{k-2}(K) \}, }
where $P_{-1}=\{0\}$ and $ \mathbb{B}_k(\partial K)=\{ v\in C^0(\partial K) : v|_e\in P_k(e)
   \ \forall e\subset K \}$.
In computation, the standard 
   interpolated virtual finite element space on $\mathcal{T}_h$ is defined by
\an{ \label{V-h} V_h = \{ v_h=\Pi_h^\nabla \tilde v \ : \ v_h|_K \in \mathbb{V}_k(K), 
  \ K\in\mathcal{T}_h; \
   \tilde v\in \tilde V_h \}, }
where $\mathbb{V}_k(K)=P_k(K)$ for the standard virtual elements (and to be defined below
   in \eqref{V-k} for the new method),
   and $v_h=\Pi_h^\nabla \tilde v$ satisfies, for all
  $w_h\in \mathbb{V}_k(K)$,
\an{\label{Pi}  \langle v_h-\tilde v, w_h\rangle_{\partial K}
   =0, \ \t{ and } \ (\nabla(v_h-\tilde v), \nabla w_h)_K=0. }
The stabilizer-free virtual element equation reads:  Find $u_h=\P \tilde u\in V_h$ such that
\an{\label{f-e} (\nabla u_h,\nabla v_h)_h = (f,v_h) 
     \quad \forall \tilde v\in \tilde V_h, \ v_h=\P \tilde v, }
where $(\nabla u_h, \nabla v_h)_h=\sum_{K\in \mathcal{T}_h} (\nabla u_h, \nabla v_h)_K$.  
But the dimension of $V_h$ is less than that of $\tilde V_h$ unless $k=1$ and on
   triangular meshes.  Thus the bilinear form in \eqref{f-e} is not coercive and 
   the equation does not have a unique solution.
A discrete stabilizer must be added to the equation \eqref{f-e}.

In order to delete the stabilizer, 
  \cite{Berrone-0} proposed to replace $\mathbb{V}_k(K)=P_k(K)$ by
   $\mathbb{V}_k(K)=P_{k+l}(K)$ in the virtual element space \eqref{V-h} for
  the case $k=1$ on polygonal meshes, where $l$ depends on the maximum number of
   edges.
 Further numerical tests and comparisons are given in \cite{Berrone-1}.
Another stabilization-free method for $k=1$ is proposed in \cite{Berrone} that
   $\mathbb{V}_k(K)=P_k(K)\cup H_l(K)$, where $H_l(K)$ is the set of 2D harmonic
   polynomials of degree $l$ or less, and $l$ depends on the maximum number of
   edges.
This is an excellent idea because the $H_l$ polynomials may enforce corerciveness while
  not destroying the gradient approximation, as they have vanishing Laplacian.
The same idea has been implemented in some other harmonic finite elements \cite{Al-Taweel,
 Sorokina1,Sorokina2}. 
But the method \cite{Berrone} 
  is shown not working for $k>3$ numerically in this paper.
Another stabilizer-free virtual element method  is proposed in \cite{Chen-H} where,
  instead of $H^1$ interpolating the virtual element functions,  
  the weak gradient (the name is used in weak Galerkin methods, and the macro-RT and
   macro-BDM are used to define the weak gradient in \cite{Ye-Z20a,Ye-Z21c,Ye-Z23c}) is defined 
   via integration by parts and macro-mixed finite elements on polygons and polyhedra.

We propose to define the interpolation space $\mathbb{V}_k(K)$ in \eqref{V-h} as
\an{ \label{V-k} \mathbb{V}_k(K)=\{ v_h \in C(K) \ : \ 
      v_h|_{K_i}\in P_k(K_i), \   K=\cup_{i=1}^3 K_i \}, }
where $K$ is split in to three triangles by connecting its barycenter with three
   vertices, cf. Figure \ref{f-h}.   We call $K$ a Hsieh-Clough-Tocher macro-triangle \cite{Clough, Sorokina2,
 Xu-Zhang,  ZhangMG, Zhang3D}.

\begin{figure}[ht] \centering 
\begin{picture}(120,90)(0,0)
 \put(1,50){$K:$} \put(40,50){$K_2$} \put(65,50){$K_1$}  \put(55,10){$K_3$}
  \put(0,0){\line(2,1){60}}\put(120,0){\line(-2,1){60}}\put(60,30){\line(0,1){60}}
  \put(0,0){\line(2,3){60}}\put(120,0){\line(-2,3){60}}\put(0,0){\line(1,0){120}}
\end{picture} \caption{A Hsieh-Clough-Tocher macro-triangle $K=\cup_{i=1}^3 K_i$}
  \label{f-h}
\end{figure}
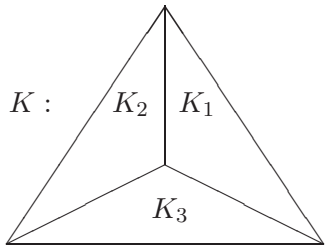

The interpolation operator $\Pi^\nabla_h$ in \eqref{Pi} is naturally
  defined by $v_h=\Pi^\nabla_h \tilde v\in \mathbb{V}_k(K)$ of \eqref{V-k} satisfying 
\an{ \label{l-e} \ad{ v_h &= \tilde v && \t{on } \ \partial K, \\
                    (\nabla v_h, \nabla w_h)_K &= (\nabla \tilde v, \nabla w_h)_K 
                      && \forall w_h \in H^1_0(K) \cap \mathbb{V}_k(K). } }

We note that a different interpolation space only changes the
   numerical quadrature formula for computing $(\nabla u_h, \nabla v_h)=
    (\nabla \Pi_h^\nabla \tilde u, \nabla \Pi_h^\nabla \tilde v)$ in the
   virtual elements equation.
An accurate calculation of local interpolation does not increase the computational cost
  once the stiffness matrix is generated.
One may see no advantage of this stabilizer-free virtual element over the Lagrange finite
   element.
But the virtual elements are mainly for polygonal and polyhedral meshes.
In \cite{Lin-Mu}, this stabilization technique is applied to 2D polygons and 
   3D polyhedra.
We separate the case of triangles because it shows the idea clearly while the polygons
  and polyhedra have natural triangular and tetrahedral subdivisions, cf. \cite{Lin-Mu}.

Eliminating the stabilizer would not only reduce computational cost, but also likely to
  improve the condition number of the resulting system.
In the numerical test, we show how the condition number of the standard virtual element
   is improved by three methods.
But the improved condition number is still worse than that of this stabilizer-free
   virtual element.

Eliminating the stabilizer would likely  utilize fully every degree of
  freedom in the discrete approximation.
Thus it often leads to discoveries of superconvergence.
In \cite{Lin-Xu}, it is shown that only this stabilizer-free $P_1$ virtual element
   converges three orders above the optimal order in $H^1$-norm, and
   two orders above the optimal order in $L^2$-norm and $L^\infty$-norm,
   when solving the Poisson equation on honeycomb meshes.

The stabilizer is eliminated first in the weak Galerkin finite element method  
  \cite{Al-Taweel-Wang,Feng-Zhang, Gao-Z,Mu1,Wang-Z20,
      Wang-Z21,Ye-Z20a,  Ye-Z20b,  Ye-Z21c},
  then in the $H(\t{div})$ finite element method \cite{Mu21,Ye-Z21a},
    in the $C^0$ or $C^{-1}$
     finite element methods for the biharmonic equation \cite{Ye-Z22b}
  and in the discontinuous Galerkin finite element method \cite{Feng-Z21,Mu23}.
It leads to two-order superconvergent WG finite elements \cite{Al-Taweel-Z21,
    Wang-Z23,Wang-Z23a, Ye-Z23e} and
  two-order superconvergent DG finite elements \cite{Ye-Z22d,Ye-Z23c} 
   for second order elliptic   equations, 
  one or two-order superconvergent WG finite elements for the Stokes equations 
\cite{Mu21b,Ye-Z21e}, 
  four-order superconvergent WG finite elements \cite{Ye-Z23d}
   and four-order  superconvergent DG finite elements \cite{Ye-Z22d,Ye-Z23e}
   for the biharmonic equation.
That is, for an example, a $P_3$ discontinuous finite element method, with the
  stabilizer-free technique, produces the same order approximate solution as a
  $C^1$-$P_7$ finite element method does, in solving a 2D biharmonic equation.

In this paper, we show
  that with the new interpolation \eqref{l-e}, the stabilizer-free virtual element equation 
   \eqref{f-e} has a unique and quasi-optimal solution, on triangular meshes.
Numerical tests on the new stabilizer-free virtual elements are performed.
Numerical comparisons are presented, with the other stabilizer-free virtual elements
   and with the standard virtual elements.

                \vskip .7cm
 
\section{The well-posedness} 

We show in this section that the stabilizer-free virtual element equation has a unique solution.

\begin{lemma} 
The interpolation operator
 $\P$ is well defined in \eqref{l-e} and it preserves $P_k$ polynomials,
\an{\label{p-p} \P \tilde v = \tilde v\quad \ \t{if } \ \tilde v\in P_k(K). }
\end{lemma}

\begin{proof} Because $\tilde v|_{\partial K}\in \mathbb{B}_k(\partial K)$, 
  $v_h$ can assume the boundary
  condition $v_h=\tilde v$ exactly on $\partial K$. The linear system of equations in \eqref{l-e}
 is a finite dimensional square system. The existence is implied by the uniqueness. 
 To show the uniqueness, we let $\tilde v=0$ in \eqref{l-e}.  
  Letting $w_h=v_h$ in \eqref{l-e}, we get
\a{ \nabla v_h =\b 0 \quad \t{ on } \ K. }
Thus $v_h=c$ is a constant on $K$.  As $v_h$ is continuous on edges,  $v_h=c$ is
a global constant on the whole domain.  By the boundary condition, we get $0=\tilde v|_
   {\partial \Omega}=v_h|_{\partial \Omega}=c$.
   Hence $v_h=0$ and \eqref{l-e} has a unique solution.  

If $\tilde v\in P_k(K)\subset \mathbb{V}_k(K)$, defined in \eqref{V-h}, then the solution   
   of \eqref{l-e} says, letting $w_h=v_h-\tilde v$, 
\a{ \nabla (v_h-\tilde v)=\b 0. }
Thus $v_h-\tilde v$ is a global constant which must be zero as it vanishes at all $\partial K$.
  \eqref{p-p} is proved. 
\end{proof}
                \vskip .7cm

\begin{lemma} 
The stabilizer-free virtual element equation \eqref{f-e} has a unique solution,
where the interpolation $\P$ is defined in \eqref{l-e}.
\end{lemma}

\begin{proof} As both $\tilde u, \tilde v \in \tilde V_h$,
  \eqref{f-e} is a finite square  system of linear equations. The uniqueness of solution 
  implies the existence. 
 To show the uniqueness, we let $f=0$ and $\tilde v=\tilde u$ in \eqref{f-e}. 
   It follows that
\a{ |\P \tilde u|_{1,h} =0. }
Thus $\P \tilde u=c$ is  constant on each $K$.  But $\P \tilde u$ is continuous on the whole
 domain.  By the boundary condition, we get $0=\P \tilde u|_
   {\partial \Omega}=c$.  That is,
\an{\label{p-u-0}  \P \tilde u=0.  }
On one triangle $K=\cup_{i=1}^3 K_i$, 
\an{\label{u-p-k} \tilde u=\P\tilde u=0 \quad \ \t{on } \ \partial K. }
Inside the triangle, by \eqref{l-e}, \eqref{p-u-0}, \eqref{u-p-k} and 
   integration by parts, we have
\an{\label{l-e1} (-\Delta \tilde u, w_h) =
     (\nabla \tilde u, \nabla w_h)=0 \quad \forall w_h\in H^1_0(K)\cap \mathbb{V}_k(K).  }
 By the space $\tilde V_h$ definition \eqref{t-V-h}, we denote
\an{\label{p-k-2} p_{k-2} = -\Delta \tilde u \in P_{k-2}(K).  }
 Let the $w_h$ in \eqref{l-e1} be
\an{\label{w-h} w_h= p_{k-2} \phi_{\b x_0}(\b x) 
   \in H^1_0(K)\cap \mathbb{V}_k(K),  }
  where $\phi_{\b x_0}(\b x)$ is the $P_1$ Lagrange basis function at
   node $\b x_0$, cf. Figure \ref{lambda}.
  That is, $\phi_{\b x_0}(\b x)|_{K_i}= \lambda_{i,0}$ is the barycentric coordinate at 
 $x_0$ on triangle $K_i$,
   i.e.,  a linear function which assumes value 1 at $\b x_0$ and vanishing on the 
   line $e_i$, cf. Figure \ref{lambda}.

\begin{figure}[ht] \centering 
\begin{picture}(120,90)(0,0)
 \put(1,50){$K:$} \put(40,50){$K_2$} \put(65,50){$K_1$}  \put(55,10){$K_3$}
 \put(62,30){$\b x_0$} \put(16,40){$e_2$} \put(96,40){$e_1$}  \put(35,2){$e_3$}
  \put(0,0){\line(2,1){60}}\put(120,0){\line(-2,1){60}}\put(60,30){\line(0,1){60}}
  \put(0,0){\line(2,3){60}}\put(120,0){\line(-2,3){60}}\put(0,0){\line(1,0){120}}
\end{picture} \caption{A Hsieh-Clough-Tocher macro-triangle $K=\cup_{i=1}^3 K_i$}
  \label{lambda}
\end{figure}
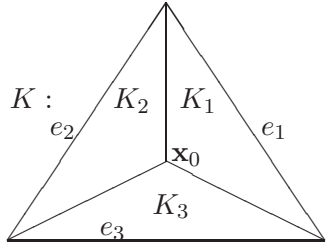

With the $w_k$ in \eqref{w-h}, we get from \eqref{l-e1} and \eqref{p-k-2} that
\a{ \int_K p_{k-2}^2 \phi_{\b x_0}(\b x) d \b x =0.  }
As $\phi_{\b x_0}(\b x)>0$ inside $K$,  it follows that
\a{ p_{k-2}^2 =0 \ \t{ and } \ p_{k-2}  =0 \ \t{ on } \ K. }
By \eqref{u-p-k} and \eqref{p-k-2}, $\Delta \tilde  u=0$ in $K$ and $\tilde u=0$ on 
   $\partial K$.  Thus, by the unique solution of the Laplace equation,  $\tilde u=0$.
  The lemma is proved. 
\end{proof}
                \vskip .7cm
   
\section{Convergence}
           
We prove the optimal order convergence of the stabilizer-free virtual element solutions in this
  section.

\begin{theorem}  Let $ u\in H^{k+1}\cap H^1_0(\Omega)$ be the exact solution of \eqref{w-e}.
 Let $u_h$ be the stabilizer-free virtual element solution of \eqref{f-e}.   It holds that
  \an{ \label{h-1}  |  u- u_h |_{1}    \le Ch^{k} | u|_{k+1}.  } 
\end{theorem}
                        
\begin{proof} Since $w_h\in V_h \subset H^1_0(\Omega) $,  we subtract \eqref{f-e} from
   \eqref{w-e} to get
\a{  (\nabla (u-u_h), \nabla  w_h)_h =0\quad \forall w_h\in   V_{h }. } 
Applying the Schwarz inequality,  it follows that
\a{    |   u- u_h|_{1}^2  
     & =  (\nabla(  u-  u_h), \nabla(  u- I_h u))\\ 
     &\le |   u- u_h|_{1} |   u- I_h u |_{1} \le Ch^{k} |u|_{k+1}
 |   u- u_h|_{1} ,} 
      where $ I_h  u$ is the Scott-Zhang interpolation on
  quasi-uniform triangulation $\mathcal{T}_h$ \cite{Scott-Zhang}. 
The proof is complete.
\end{proof}

To get the optimal order $L^2$ error bound,  we need a full regularity of the dual
   equation that the solution of the equation, 
\a{ -\Delta w &= g \quad \t{ in } \ \Omega, \\
            w &=0  \quad \t{ on } \ \partial \Omega, }
satisfies
\an{\label{r} |w|_2 \le C \|g\|_0.  }

\begin{theorem}  Let $ u\in H^{k+1}\cap H^1_0(\Omega)$ be the exact solution of \eqref{w-e}.
 Let $u_h$ be the stabilizer-free virtual element solution of \eqref{f-e}.  Assuming
   \eqref{r}, it holds that
  \a{   \|  u- u_h \|_{0}    \le Ch^{k+1} | u|_{k+1}.  } 
\end{theorem}
                        
\begin{proof} Let $w\in H^2(\Omega)\cap H^1_0(\Omega)$ be the dual solution,
  \an{ \label{d2}
    ( \nabla w, \nabla v) &=(u-u_h, v), 
       \ \forall v \in H^1_0(\Omega).}
Thus, by \eqref{d2}, \eqref{r} and \eqref{h-1},  we get
\a{ \|u-u_h\|_0^2 &=(\nabla w, \nabla (u-u_h) ) = 
(\nabla (w-w_h), \nabla (u-u_h) ) \\
  & \le C h |w|_2  h^{k} |u|_{k+1} \le C h^{k+1} |u|_{k+1}\|u-u_h\|_0, }
where $w_h$ is the virtual element solution to the equation \eqref{d2}.
We obtain the $L^2$ error bound. 
\end{proof}

                \vskip .7cm

\section{Numerical test}

We solve numerically the Poisson equation \eqref{p-e}
   on the domain $\Omega=(0,1)\times(0,1)$, where an exact solution is chosen as
\an{\label{s-1} u(x,y)=\sin (\pi x) \sin (\pi y).  }

     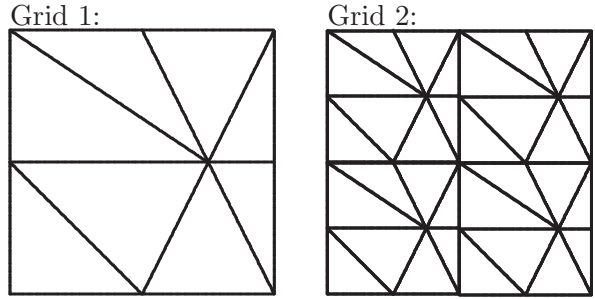
\begin{figure}[ht] \setlength\unitlength{1pt}\begin{center}
    \begin{picture}(220,110)(0,0)
     \def\mc{\begin{picture}(  100.,  100.)(  0.,  0.) 
     \def\la{\circle*{0.3}} 
     \multiput(   0.00,   0.00)(   0.250,   0.000){200}{\la}
     \multiput(  50.00,   0.00)(   0.250,   0.000){200}{\la}
     \multiput( 100.00,   0.00)(   0.000,   0.250){200}{\la}
     \multiput( 100.00,  50.00)(   0.000,   0.250){200}{\la}
     \multiput(   0.00, 100.00)(   0.250,   0.000){200}{\la}
     \multiput(  50.00, 100.00)(   0.250,   0.000){200}{\la}
     \multiput(   0.00,   0.00)(   0.000,   0.250){200}{\la}
     \multiput(   0.00,  50.00)(   0.000,   0.250){200}{\la}
     \multiput(  50.00,   0.00)(   0.112,   0.224){223}{\la}
     \multiput(  75.00,  50.00)(   0.250,   0.000){100}{\la}
     \multiput(  50.00, 100.00)(   0.112,  -0.224){223}{\la}
     \multiput(   0.00,  50.00)(   0.250,   0.000){300}{\la}
     \multiput(   0.00,  50.00)(   0.177,  -0.177){282}{\la}
     \multiput(  75.00,  50.00)(   0.112,  -0.224){223}{\la}
     \multiput(  75.00,  50.00)(   0.112,   0.224){223}{\la}
     \multiput(   0.00, 100.00)(   0.208,  -0.139){360}{\la} 
 \end{picture}}
      
\put(0,102){Grid 1:} \put(120,102){Grid 2:}
    \put(0,0){\mc}
      \put(120,0){\setlength\unitlength{0.5pt}\begin{picture}(90,90)(0,0)
    \multiput(0,0)(100,0){2}{\multiput(0,0)(0,100){2}{\mc}} \end{picture}}
   % \put(220,0){\setlength\unitlength{0.25pt}\begin{picture}(90,90)(0,0)
   % \multiput(0,0)(100,0){4}{\multiput(0,0)(0,100){4}{\mc}} \end{picture}}
    \end{picture}
 \caption{The first two levels of grids in the computation of
  Tables \ref{t-1}--\ref{t-4}. }\label{grid-8} 
    \end{center} \end{figure}

The computation is done first on a family of slightly irregular
    triangular meshes shown in Figure \ref{grid-8}.

In Table \ref{t-1}, we first list the errors and the computed order of convergence
  for the $P_1$ and $P_2$ stabilizer-free virtual elements \eqref{V-h} with the
   Hsieh-Clough-Tocher macro-triangle interpolation space $\mathbb{V}_k$ in \eqref{V-k}.
Optimal orders are achieved for both elements and in both $L^2$ and $H^1$ norms.
Here we use $\P u$ instead of $u$ to check the error so that we can detect 
  possible superconvergence.
At the bottom of Table \ref{t-1},  we test a method of \cite{Berrone} where the
  interpolation space is $P_2\cup H_3$, i.e., enriching the $P_2(K)$ space by
  two harmonic $P_3$ polynomials.
The method is not proved yet.  The numerical test shows it works well, producing
  optimal order errors.

\begin{table}[ht]
  \centering  \renewcommand{\arraystretch}{1.2}
  \caption{Error profile on Figure \ref{grid-8} meshes  for  \eqref{s-1}. }
  \label{t-1}
\begin{tabular}{c|cc|cc}
\hline
Grid &   $\|\Pi^\nabla_h u-u_h\|_{0}$  &  $O(h^r)$ 
   &   $|\Pi^\nabla_h u-u_h|_{1}$ & $O(h^r)$ 
  \\ \hline 
    &  \multicolumn{4}{c}{ By the $P_1$ virtual element with HTC interpolation. } \\
\hline  
 7&   0.8208E-04 & 2.00&   0.7621E-02 & 1.00 \\
 8&   0.2052E-04 & 2.00&   0.3817E-02 & 1.00 \\
 9&   0.5129E-05 & 2.00&   0.1911E-02 & 1.00 \\
\hline
    &  \multicolumn{4}{c}{ By the $P_2$ virtual element with HTC interpolation. } \\
\hline  
 7&   0.7537E-07 & 3.00&   0.6668E-04 & 1.99 \\
 8&   0.9423E-08 & 3.00&   0.1670E-04 & 2.00 \\
 9&   0.1178E-08 & 3.00&   0.4179E-05 & 2.00 \\
\hline  
 &  \multicolumn{4}{c}{ By the $P_2$ virtual element with 2 h.p. \cite{Berrone}. }  \\
\hline  
 7&   0.8627E-07 & 3.00&   0.7861E-04 & 2.00 \\
 8&   0.1079E-07 & 3.00&   0.1968E-04 & 2.00 \\
 9&   0.1350E-08 & 3.00&   0.4925E-05 & 2.00 \\ 
\hline
    \end{tabular}%
\end{table}%

In Table \ref{t-2}, we first list the errors and the computed order of convergence
  for the $P_3$ stabilizer-free virtual elements, $k=3$ in \eqref{V-h} with the
   Hsieh-Clough-Tocher macro-triangle interpolation space $\mathbb{V}_k$ in \eqref{V-k}.
Optimal orders are achieved for the element in both $L^2$ and $H^1$ norms.
At the bottom of Table \ref{t-2},  we test a method of \cite{Berrone} where the
  interpolation space is $P_3\cup H_5$, i.e., enriching the $P_3(K)$ polynomial space by
  four harmonic $P_4$ and $P_5$ polynomials.
The method is not proved yet.  The numerical test shows it works well, producing
  optimal order errors.

\begin{table}[ht]
  \centering  \renewcommand{\arraystretch}{1.2}
  \caption{Error profile on Figure \ref{grid-8} meshes  for  \eqref{s-1}. }
  \label{t-2}
\begin{tabular}{c|cc|cc}
\hline
Grid &   $\|\Pi^\nabla_h u-u_h\|_{0}$  &  $O(h^r)$ 
   &   $|\Pi^\nabla_h u-u_h|_{1}$ & $O(h^r)$ 
  \\ \hline 
    &  \multicolumn{4}{c}{ By the $P_3$ virtual element with HTC interpolation. } \\
\hline  
 6&   0.4664E-08 & 4.00&   0.3068E-05 & 2.99 \\
 7&   0.2916E-09 & 4.00&   0.3844E-06 & 3.00 \\
 8&   0.1824E-10 & 4.00&   0.4811E-07 & 3.00 \\\hline 
 &  \multicolumn{4}{c}{ By the $P_3$ virtual element with 8 h.p. \cite{Berrone}. }  \\
\hline  
 6&   0.5676E-08 & 4.01&   0.3392E-05 & 2.99 \\
 7&   0.3542E-09 & 4.00&   0.4257E-06 & 2.99 \\
 8&   0.2222E-10 & 3.99&   0.5372E-07 & 2.99 \\
\hline  
    \end{tabular}%
\end{table}%

In Table \ref{t-3}, we first list the errors and the computed orders of convergence
  for the $P_4$ stabilizer-free virtual elements, $k=4$ in \eqref{V-h} with the
   Hsieh-Clough-Tocher macro-triangle interpolation space $\mathbb{V}_k$ in \eqref{V-k}.
Optimal orders are achieved for the element in both $L^2$ and $H^1$ norms.
Here the computer accuracy is exhausted when computing the last grid solution.
At the bottom of Table \ref{t-3},  we test a method of \cite{Berrone} where the
  interpolation space is $P_4\cup H_{10}$, i.e., enriching the $P_4(K)$
   polynomial space by
  12 harmonic $P_5$, $P-6$, $P_7$, $P_8$, $P_9$ and $P_{10}$ polynomials.
Since the method does not work,  we tested by adding more harmonic polynomials 
  until  the error can not be reduced anymore.

\begin{table}[ht]
  \centering  \renewcommand{\arraystretch}{1.2}
  \caption{Error profile on Figure \ref{grid-8} meshes  for  \eqref{s-1}. }
  \label{t-3}
\begin{tabular}{c|cc|cc}
\hline
Grid &   $\|\Pi^\nabla_h u-u_h\|_{0}$  &  $O(h^r)$ 
   &   $|\Pi^\nabla_h u-u_h|_{1}$ & $O(h^r)$ 
  \\ \hline  
    &  \multicolumn{4}{c}{ By the $P_4$ virtual element with HTC interpolation. } \\
\hline  
 5&   0.7612E-09 & 5.00&   0.3811E-06 & 3.99 \\
 6&   0.2374E-10 & 5.00&   0.2389E-07 & 4.00 \\
 7&   0.7616E-12 & --- &   0.1495E-08 & 4.00 \\
\hline   
 &  \multicolumn{4}{c}{ By the $P_4$ virtual element with 12 h.p. \cite{Berrone}. }  \\
\hline  
 4&   0.4589E-07 & 5.01&   0.7336E-05 & 3.94 \\
 5&   0.1525E-08 & 4.91&   0.5830E-06 & 3.65 \\
 6&   0.9493E-10 & 4.01&   0.9239E-07 & 2.66 \\
\hline
    \end{tabular}%
\end{table}%

In Table \ref{t-4}, we first list the errors and the computed orders of convergence
  for the $P_5$ stabilizer-free virtual elements, $k=5$ in \eqref{V-h} with the
   Hsieh-Clough-Tocher macro-triangle interpolation space $\mathbb{V}_k$ in \eqref{V-k}.
Optimal orders are achieved for the element in both $L^2$ and $H^1$ norms.
At the middle of Table \ref{t-4},  we test a method of \cite{Berrone} where the
  interpolation space is $P_5\cup H_{10}$, i.e., enriching the $P_5(K)$
   polynomial space by
  10 harmonic   $P_6$, $P_7$, $P_8$, $P_9$ and $P_{10}$ polynomials.
Since the method does not work,  we tested by adding more harmonic polynomials 
  until  the error can not be reduced anymore.
Comparing to the last case,  the convergent order deteriorates a lot.
At the bottom of Table \ref{t-4}, we 
  list the errors and the computed orders of convergence
  for the $P_6$ stabilizer-free virtual elements, $k=6$ in \eqref{V-h} with the
   Hsieh-Clough-Tocher macro-triangle interpolation space $\mathbb{V}_k$ in \eqref{V-k}.
Optimal orders are achieved for the element in both $L^2$ and $H^1$ norms.

\begin{table}[ht]
  \centering  \renewcommand{\arraystretch}{1.2}
  \caption{Error profile on Figure \ref{grid-8} meshes  for  \eqref{s-1}. }
  \label{t-4}
\begin{tabular}{c|cc|cc}
\hline
Grid &   $\|\Pi^\nabla_h u-u_h\|_{0}$  &  $O(h^r)$ 
   &   $|\Pi^\nabla_h u-u_h|_{1}$ & $O(h^r)$ 
  \\ \hline 
    &  \multicolumn{4}{c}{ By the $P_5$ virtual element with HTC interpolation. } \\
\hline  
 3&   0.3339E-07 & 6.01&   0.5275E-05 & 4.96 \\
 4&   0.5170E-09 & 6.01&   0.1665E-06 & 4.99 \\
 5&   0.8033E-11 & 6.01&   0.5223E-08 & 4.99 \\\hline 
 &  \multicolumn{4}{c}{ By the $P_5$ virtual element with 10 h.p. \cite{Berrone}. }  \\
\hline  
 2&   0.4262E-05 & 5.18&   0.2158E-03 & 4.24 \\
 3&   0.1566E-06 & 4.77&   0.2319E-04 & 3.22 \\
 4&   0.1819E-07 & 3.11&   0.5702E-05 & 2.02 \\
\hline 
    &  \multicolumn{4}{c}{ By the $P_6$ virtual element with HTC interpolation. } \\
\hline  
 2&   0.1696E-06 & 7.30&   0.1576E-04 & 6.19 \\
 3&   0.1320E-08 & 7.01&   0.2515E-06 & 5.97 \\
 4&   0.1025E-10 & 7.01&   0.3956E-08 & 5.99 \\
\hline   
    \end{tabular}%
\end{table}%

The next part of computation is done on the uniform triangular meshes, shown
  as in Figure \ref{t-grid}.  This is mainly for detecting possible superconvergence.

     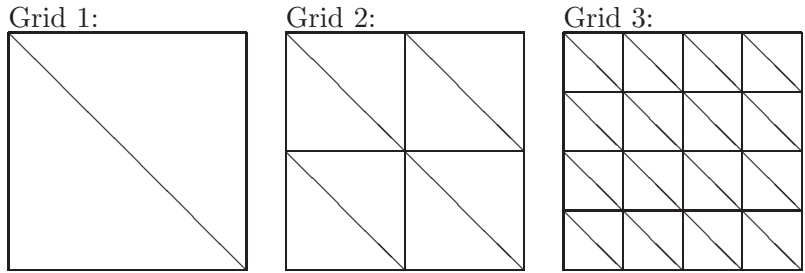
\begin{figure}[ht] \setlength\unitlength{1pt}\begin{center}
    \begin{picture}(300,102)(0,0)
     \def\mc{\begin{picture}(90,90)(0,0)
       \put(0,0){\line(1,0){90}} \put(0,90){\line(1,0){90}}
      \put(0,0){\line(0,1){90}}  \put(90,0){\line(0,1){90}} % \put(0,0){\line(1,1){90}} 
     \put(0,90){\line(1,-1){90}}
      \end{picture}}

\put(0,92){Grid 1:} \put(105,92){Grid 2:}\put(210,92){Grid 3:}
    \put(0,0){\mc}
      \put(105,0){\setlength\unitlength{0.5pt}\begin{picture}(90,90)(0,0)
    \put(0,0){\mc}\put(90,0){\mc}\put(0,90){\mc}\put(90,90){\mc}\end{picture}}
      \put(210,0){\setlength\unitlength{0.25pt}\begin{picture}(90,90)(0,0)
    \multiput(0,0)(90,0){4}{\multiput(0,0)(0,90){4}{\mc}} \end{picture}}
    \end{picture}
 \caption{The first three levels of grids for the computation in Tables \ref{t-21}--\ref{t-23}. }\label{t-grid} 
    \end{center} \end{figure}

In Table \ref{t-21}, we first list the errors and the computed orders of convergence
  for the $P_1$ stabilizer-free virtual elements, $k=1$ in \eqref{V-h} with the
   Hsieh-Clough-Tocher macro-triangle interpolation space $\mathbb{V}_k$ in \eqref{V-k}.
Optimal orders are achieved for the element in both $L^2$ and $H^1$ norms.
In fact, we have one-order superconvergence in $H^1$ semi-norm.
At the bottom of Table \ref{t-21}, we 
  list the errors and the computed orders of convergence
  for the $P_1$ standard virtual elements, $k=1$ in \eqref{V-h}.
Optimal orders are achieved for the element in both $L^2$ and $H^1$ norms.
Again, we have one-order $H^1$ superconvergence for this element.
Comparing the errors,  the new method is slightly better which is understandable as
  their interpolation space $P_1(K)$ is a subspace of our
   interpolation space $\mathbb{V}_1=C(K)\cap \cup_{i=1}^3 P_1(K_i)$.

\begin{table}[ht]
  \centering  \renewcommand{\arraystretch}{1.2}
  \caption{Error profile on Figure \ref{t-grid} meshes  for  \eqref{s-1}. }
  \label{t-21}
\begin{tabular}{c|cc|cc}
\hline
Grid &   $\|\Pi^\nabla_h u-u_h\|_{0}$  &  $O(h^r)$ 
   &   $|\Pi^\nabla_h u-u_h|_{1}$ & $O(h^r)$ 
  \\ \hline 
    &  \multicolumn{4}{c}{ By the $P_1$ SF virtual element with HTC interpolation. } \\
\hline  
 7&   0.3032E-03 & 2.00&   0.1382E-02 & 2.00 \\
 8&   0.7586E-04 & 2.00&   0.3457E-03 & 2.00 \\
 9&   0.1897E-04 & 2.00&   0.8644E-04 & 2.00 \\
 \hline 
&  \multicolumn{4}{c}{ By the standard $P_1$ virtual element \cite{Beirao}. }  \\
\hline   
 6&   0.1210E-02 & 1.98&   0.5518E-02 & 1.99 \\
 7&   0.3032E-03 & 2.00&   0.1382E-02 & 2.00 \\
 8&   0.7586E-04 & 2.00&   0.3457E-03 & 2.00 \\
\hline  
    \end{tabular}%
\end{table}%

In Table \ref{t-22}, we first list the errors and the computed orders of convergence
  for the $P_2$ stabilizer-free virtual elements, $k=2$ in \eqref{V-h} with the
   Hsieh-Clough-Tocher macro-triangle interpolation space $\mathbb{V}_k$ in \eqref{V-k}.
Optimal orders are achieved for the element in both $L^2$ and $H^1$ norms.
Unlike the traditional $P_2$ finite element, we do not have any superconvergence.
Comparing to the traditional $P_2$ finite element,  we compute a solution in a larger
  vector space but get a worse result.  This is because the added Hsieh-Clough-Tocher
  macro-bubbles destroy the symmetry of $P_2$ finite element equations on uniform 
   triangular meshes.  
At the bottom of Table \ref{t-22}, we 
  list the errors and the computed orders of convergence
  for the standard $P_2$  virtual elements, $k=2$ in \eqref{V-h}.
Optimal orders are achieved for the element in both $L^2$ and $H^1$ norms.
Comparing the two errors,  the new method is much better which is understandable as
  there is no stabilizer here.

\begin{table}[ht]
  \centering  \renewcommand{\arraystretch}{1.2}
  \caption{Error profile on Figure \ref{t-grid} meshes  for  \eqref{s-1}. }
  \label{t-22}
\begin{tabular}{c|cc|cc}
\hline
Grid &   $\|\Pi^\nabla_h u-u_h\|_{0}$  &  $O(h^r)$ 
   &   $|\Pi^\nabla_h u-u_h|_{1}$ & $O(h^r)$ 
  \\ \hline 
    &  \multicolumn{4}{c}{ By the $P_2$ SF virtual element with HTC interpolation. } \\
\hline  
 7&   0.5604E-07 & 3.72&   0.2302E-04 & 2.48 \\
 8&   0.5261E-08 & 3.41&   0.4990E-05 & 2.21 \\
 9&   0.5901E-09 & 3.16&   0.1194E-05 & 2.06 \\
\hline   
 &  \multicolumn{4}{c}{ By the standard $P_2$ virtual element  \cite{Beirao}. }  \\
\hline   
 6&   0.1885E-05 & 3.15&   0.3225E-03 & 2.08 \\
 7&   0.2285E-06 & 3.04&   0.7958E-04 & 2.02 \\
 8&   0.2833E-07 & 3.01&   0.1983E-04 & 2.00 \\
\hline 
    \end{tabular}%
\end{table}%

In Table \ref{t-23}, we list the errors and the computed orders of convergence
  for the $P_3$, $P_4$, $P_5$ and $P_6$   
     stabilizer-free virtual elements, $k=3, 4, 5, $ or $6$ in \eqref{V-h} with the
   Hsieh-Clough-Tocher macro-triangle interpolation space $\mathbb{V}_k$ in \eqref{V-k}.
Optimal orders are achieved for all the elements in both $L^2$ and $H^1$ norms.
Comparing the errors of same virtual elements, 
    the uniform triangular meshes are much better than the triangular
   meshes shown in Figure \ref{grid-8}.

\begin{table}[ht]
  \centering  \renewcommand{\arraystretch}{1.2}
  \caption{Error profile on Figure \ref{grid-8} meshes  for  \eqref{s-1}. }
  \label{t-23}
\begin{tabular}{c|cc|cc}
\hline
Grid &   $\|\Pi^\nabla_h u-u_h\|_{0}$  &  $O(h^r)$ 
   &   $|\Pi^\nabla_h u-u_h|_{1}$ & $O(h^r)$ 
  \\ \hline 
    &  \multicolumn{4}{c}{ By the $P_3$ SF virtual element with HTC interpolation. } \\
\hline  
 6&   0.4906E-07 & 3.98&   0.1628E-04 & 2.98 \\
 7&   0.3087E-08 & 3.99&   0.2048E-05 & 2.99 \\
 8&   0.1935E-09 & 4.00&   0.2567E-06 & 3.00 \\
 \hline 
    &  \multicolumn{4}{c}{ By the $P_4$ SF virtual element with HTC interpolation. } \\
\hline  
 5&   0.1038E-07 & 5.03&   0.3216E-05 & 4.00 \\
 6&   0.3216E-09 & 5.01&   0.2011E-06 & 4.00 \\
 7&   0.1001E-10 & 5.01&   0.1258E-07 & 4.00 \\
 \hline  
    &  \multicolumn{4}{c}{ By the $P_5$ SF virtual element with HTC interpolation. } \\
\hline  
 5&   0.2035E-09 & 6.00&   0.7749E-07 & 4.98 \\
 6&   0.3173E-11 & 6.00&   0.2433E-08 & 4.99 \\
 7&   0.5018E-13 & 5.98&   0.7612E-10 & 5.00 \\
\hline   
    &  \multicolumn{4}{c}{ By the $P_6$ SF virtual element with HTC interpolation. } \\
\hline  
 3&   0.5536E-07 & 6.99&   0.5949E-05 & 5.94 \\
 4&   0.4294E-09 & 7.01&   0.9376E-07 & 5.99 \\
 5&   0.3333E-11 & 7.01&   0.1468E-08 & 6.00 \\
\hline    
    \end{tabular}%
\end{table}%

We would compare more the stabilizer-free virtual element 
with the standard virtual elements of   
  \cite{Beirao}.
The standard $H^1$ interpolation is defined by $\Pi_h^1 \tilde u\in P_k(K)$ such that
\an{\label{s-i} \ad{ (\nabla\Pi_h^1 \tilde u, \nabla p)_K &=-(\tilde u, \Delta p)_K
            +\langle \tilde u, \nabla p\cdot \b n \rangle_{\partial K} 
               \quad \forall p\in P_k(K)\setminus P_0(K), \\
              \langle \Pi_h^1 \tilde u, p \rangle_{\partial K} 
          &=\langle \tilde u, p \rangle_{\partial K} \quad \forall p\in P_0(K), } }
where $\b n$ is the unit outer normal vector. $\tilde u$ is defined by the degrees of
   freedom on triangle $K=\b x_1\b x_2\b x_3$, $\{F_i, i=1,\dots, N_K\}$, cf. \cite{Beirao},
\an{\label{s-d} F_i(\tilde u)=\begin{cases} 
    \tilde u(\b x_j), &j=1,2,3, \\
    \tilde u( \frac{ l\b x_j+(k-l)\b x_{\t{\tiny mod}(j,3)+1} }k), &l=1,\dots,k-1; \ j=1,2,3, \\
    \frac{\int_K \tilde u (x-x_0)^j(y-y_0)^l d\b x}
         {\int_K  1  d\b x},  & 0\le j+l \le k-2, \end{cases} }
 where $\b x_0=( x_0, y_0)$ is the barycenter of $K$.   The standard stabilizer 
in \cite{Beirao} is defined as
\an{\label{s-s} S(\tilde u-\Pi^1_h \tilde u,\tilde v-\Pi^1_h \tilde v)_K
    =\sum_{i=1}^{N_K} F_i(\tilde u-\Pi^1_h \tilde u)F_i(\tilde v-\Pi^1_h \tilde v).  }

In Table \ref{t-n-1}, we compute the solution \eqref{s-1} again by the HTC-interpolating VM
   method and by the standard virtual element method \cite{Beirao} defined by 
   \eqref{s-i}, \eqref{s-d} and \eqref{s-s}.
We first read the condition number $\kappa_2(A)$ in $l^2$ norm for the stiffness matrix $A$
   in Table \ref{t-n-1}.
We can see the condition number is huge for the \cite{Beirao} $P_3$ VM, 
    with \eqref{s-d} and \eqref{s-s}.
It would make a direct solver fail on higher level meshes.
When we read the stiffness matrix of the \cite{Beirao} $P_3$ VM,
  we find the term for the basis function associate with $(x-x_0)^1$-moment is
   much bigger than that with $(x-x_0)^0$-moment.
Therefore we replace the degrees of freedom \eqref{s-d} by a better scaled set,
\an{\label{s-d-1} F_i(\tilde u)=\begin{cases} 
    \tilde u(\b x_j), &j=1,2,3, \\
    \tilde u( \frac{ l\b x_j+(k-l)\b x_{\t{\tiny mod}(j,3)+1} }k), &l=1,\dots,k-1; \ j=1,2,3, \\
    \frac{\int_K \tilde u (x-x_0)^j(y-y_0)^l d\b x}
         {(\int_K  [(x-x_0)^j(y-y_0)^l]^2  d\b x)^{1/2}},  & 0\le j+l \le k-2. \end{cases} }
The condition number of the \cite{Beirao} $P_3$ VM (with \eqref{s-d-1} and \eqref{s-s}) is
   improved, see the third part of Table \ref{t-n-1}. 
When we read the new stiffness matrix of the \cite{Beirao} $P_3$ VM,
  we find the term for the basis function associate with $(x-x_0)^0$-moment is
   much bigger than that with the degree of freedom $\tilde u(\b x_1)$.
Therefore we scale the degrees of freedom \eqref{s-d} again,
\an{\label{s-d-2} F_i(\tilde u)=\begin{cases} 
    \tilde u(\b x_j), &j=1,2,3, \\
    \tilde u( \frac{ l\b x_j+(k-l)\b x_{\t{\tiny mod}(j,3)+1} }k), &l=1,\dots,k-1; \ j=1,2,3, \\
    \frac{10 \int_K \tilde u (x-x_0)^j(y-y_0)^l d\b x}
         {(\int_K  [(x-x_0)^j(y-y_0)^l]^2  d\b x)^{1/2} },  & 0\le j+l \le k-2. \end{cases} }
The condition number is reduced again, seen in the fourth part of Table \ref{t-n-1}.
This is nearly the best we can do about the conditioning.
It is better only at the first level than the new virtual element's condition number.
Supposedly, changing the basis does not change the solution.
But here the error of the solution in the fourth part of Table \ref{t-n-1}
  is changed (smaller).
It indicates that the bad condition number of the \cite{Beirao} $P_3$ VM does 
   increase round-off errors.

\begin{table}[ht]
  \centering  \renewcommand{\arraystretch}{1.2}
  \caption{Error profile on Figure \ref{grid-8} meshes  for  \eqref{s-1}. }
  \label{t-n-1}
\begin{tabular}{c|cc|cc|c}
\hline
Grid &   $\|\Pi^\nabla_h u-u_h\|_{0}$  &  $O(h^r)$ 
   &   $|\Pi^\nabla_h u-u_h|_{1}$ & $O(h^r)$ & $\kappa_2(A)$
  \\ \hline 
    &  \multicolumn{5}{c}{ By the $P_3$ virtual element with HTC interpolation. } \\
\hline  
 1&   0.3607E-02 & 0.00&   0.6845E-01 & 0.00&   0.6783E+03 \\
 2&   0.2967E-03 & 3.60&   0.1129E-01 & 2.60&   0.7205E+03 \\
 3&   0.1890E-04 & 3.97&   0.1510E-02 & 2.90&   0.7273E+03 \\
 4&   0.1190E-05 & 3.99&   0.1934E-03 & 2.96&   0.1204E+04 \\
 5&   0.7456E-07 & 4.00&   0.2443E-04 & 2.98&   0.4812E+04 \\
  \hline 
    &  \multicolumn{5}{c}{ By the \cite{Beirao} $P_3$ VM, with \eqref{s-d} and \eqref{s-s}. } \\ 
  \hline 
 1&   0.4093E-02 & 0.00&   0.7393E-01 & 0.00&   0.7286E+05 \\
 2&   0.4283E-03 & 3.26&   0.1361E-01 & 2.44&   0.3420E+06 \\
 3&   0.3187E-04 & 3.75&   0.2224E-02 & 2.61&   0.1353E+07 \\
 4&   0.3037E-05 & 3.39&   0.4359E-03 & 2.35&   0.5400E+07 \\
\hline   
   &  \multicolumn{5}{c}{ By the \cite{Beirao} $P_3$ VM, with \eqref{s-d-1} and \eqref{s-s}. } \\ 
  \hline  
 1&   0.4093E-02 & 0.00&   0.7393E-01 & 0.00&   0.6325E+04 \\
 2&   0.4283E-03 & 3.26&   0.1361E-01 & 2.44&   0.2749E+05 \\
 3&   0.3187E-04 & 3.75&   0.2224E-02 & 2.61&   0.1070E+06 \\
 4&   0.3037E-05 & 3.39&   0.4359E-03 & 2.35&   0.4253E+06 \\  
 \hline   
   &  \multicolumn{5}{c}{ By the \cite{Beirao} $P_3$ VM, with \eqref{s-d-2} and \eqref{s-s}. } \\ 
 \hline  
 1&   0.4089E-02 & 0.00&   0.7384E-01 & 0.00&   0.3581E+03 \\  
 2&   0.4278E-03 & 3.26&   0.1359E-01 & 2.44&   0.1530E+04 \\  
 3&   0.3176E-04 & 3.75&   0.2216E-02 & 2.62&   0.6073E+04 \\  
 4&   0.3018E-05 & 3.40&   0.4332E-03 & 2.35&   0.2424E+05 \\  
 \hline
   \end{tabular}%
\end{table}%

The \cite{Beirao} $P_3$ virtual element solution does not converge at the correct order 
   in Table \ref{t-n-1}.  This problem does not happen to the $P_1$ and $P_2$ VM solutions,
  cf. Tables \ref{t-21} and \ref{t-22}.
Thus we increase the power of the stabilizer $S(\cdot, \cdot)_K$ in \eqref{s-s} by a scaling,
\an{\label{s-s-1} S(\tilde u-\Pi^1_h u,\tilde v-\Pi^1_h v)_K
    =h^{\alpha} \sum_{i=1}^{N_K} F_i(\tilde u-\Pi^1_h u)F_i(\tilde v-\Pi^1_h v),  }
where $\alpha$ is to be specified, depending on the polynomial degree $k$.
  In Table \ref{t-n-2}, the error and the computed order of convergence are listed
   for the \cite{Beirao} $P_3$ VM, with stabilizer's $\alpha=0$ and $-1$ in \eqref{s-s-1}.
We can see, for the latter, the method can converge at the optimal order.
The errors of the stabilizer-free VM are slightly smaller, see Table \ref{t-2}.  

\begin{table}[ht]
  \centering  \renewcommand{\arraystretch}{1.2}
  \caption{Error profile on Figure \ref{grid-8} meshes  for  \eqref{s-1}. }
  \label{t-n-2}
\begin{tabular}{c|cc|cc}
\hline
Grid &   $\|\Pi^\nabla_h u-u_h\|_{0}$  &  $O(h^r)$ 
   &   $|\Pi^\nabla_h u-u_h|_{1}$ & $O(h^r)$  
  \\ \hline   
   &  \multicolumn{4}{c}{ By the \cite{Beirao} $P_3$ VM, with \eqref{s-d-2} and \eqref{s-s-1},
         $\alpha=0$. } \\ 
 \hline  
 1&   0.4089E-02 & 0.00&   0.7384E-01 & 0.00 \\
 2&   0.4278E-03 & 3.26&   0.1359E-01 & 2.44 \\
 3&   0.3176E-04 & 3.75&   0.2216E-02 & 2.62 \\
 4&   0.3018E-05 & 3.40&   0.4332E-03 & 2.35 \\ 
 \hline  
   &  \multicolumn{4}{c}{ By the \cite{Beirao} $P_3$ VM, with \eqref{s-d-2} and \eqref{s-s-1},
         $\alpha=-1$. } \\ 
 \hline  
 1&   0.4093E-02 & 0.00&   0.7393E-01 & 0.00 \\
 2&   0.4069E-03 & 3.33&   0.1261E-01 & 2.55 \\
 3&   0.2529E-04 & 4.01&   0.1681E-02 & 2.91 \\
 4&   0.1583E-05 & 4.00&   0.2172E-03 & 2.95 \\
 5&   0.9949E-07 & 3.99&   0.2765E-04 & 2.97 \\
 6&   0.6245E-08 & 3.99&   0.3491E-05 & 2.99 \\
 \hline
   \end{tabular}%
\end{table}%

 In Table \ref{t-n-3}, the error and the computed order of convergence are listed
   for the \cite{Beirao} $P_4$ VM, with stabilizer's $\alpha=0$, $-1$ and $-2$ in \eqref{s-s-1}.
We can see that the standard stabilizer \eqref{s-s} does not work when $k=4$.
We can see, for the last $\alpha=-2$ (depending on polynomial degree $k=4$ here),
    the method does converge at the optimal order, in  Table \ref{t-n-3}.
The errors of the stabilizer-free VM are slightly smaller, see Table \ref{t-3}.

\begin{table}[ht]
  \centering  \renewcommand{\arraystretch}{1.2}
  \caption{Error profile on Figure \ref{grid-8} meshes  for  \eqref{s-1}. }
  \label{t-n-3}
\begin{tabular}{c|cc|cc}
\hline
Grid &   $\|\Pi^\nabla_h u-u_h\|_{0}$  &  $O(h^r)$ 
   &   $|\Pi^\nabla_h u-u_h|_{1}$ & $O(h^r)$  
  \\ \hline   
   &  \multicolumn{4}{c}{ By the \cite{Beirao} $P_4$ VM, with \eqref{s-d-2} and \eqref{s-s-1},
         $\alpha=0$. } \\ 
 \hline  
 1&   0.1655E-02 & 0.00&   0.3146E-01 & 0.00 \\
 2&   0.8753E-04 & 4.24&   0.3987E-02 & 2.98 \\
 3&   0.1025E-04 & 3.09&   0.9204E-03 & 2.12 \\
 4&   0.1298E-05 & 2.98&   0.2294E-03 & 2.00 \\
 \hline  
   &  \multicolumn{4}{c}{ By the \cite{Beirao} $P_4$ VM, with \eqref{s-d-2} and \eqref{s-s-1},
         $\alpha=-1$. } \\ 
 \hline  
 1&   0.1655E-02 & 0.00&   0.3146E-01 & 0.00 \\
 2&   0.6940E-04 & 4.58&   0.3162E-02 & 3.31 \\
 3&   0.5106E-05 & 3.76&   0.4866E-03 & 2.70 \\
 4&   0.3993E-06 & 3.68&   0.7665E-04 & 2.67 \\
 \hline
   &  \multicolumn{4}{c}{ By the \cite{Beirao} $P_4$ VM, with \eqref{s-d-2} and \eqref{s-s-1},
         $\alpha=-1$. } \\ 
 \hline  
 1&   0.1655E-02 & 0.00&   0.3146E-01 & 0.00 \\
 2&   0.5695E-04 & 4.86&   0.2511E-02 & 3.65 \\
 3&   0.2260E-05 & 4.66&   0.2115E-03 & 3.57 \\
 4&   0.7827E-07 & 4.85&   0.1489E-04 & 3.83 \\
 5&   0.2514E-08 & 4.96&   0.9577E-06 & 3.96 \\
 6&   0.7946E-10 & 4.98&   0.6012E-07 & 3.99 \\

 \hline
   \end{tabular}%
\end{table}%

\section{Ethical Statement}

\subsection{Compliance with Ethical Standards} { \ }

   The submitted work is original and is not published elsewhere in any form or language.

\subsection{Funding } { \ }

Xuejun Xu was supported by National Natural Science Foundation of China (Grant
No. 12071350), Shanghai Municipal Science and Technology Major Project No.
2021SHZDZX0100, and Science and Technology Commission of Shanghai Municipality.

\subsection{Conflict of Interest} { \ }

  There is no potential conflict of interest .

\subsection{Ethical approval} { \ }

  This article does not contain any studies involving animals.
This article does not contain any studies involving human participants.
  
\subsection{Informed consent}  { \ }

This research does not have any human participant.  

\subsection{Availability of supporting data } { \ }

This research does not use any external or author-collected data.

\subsection{Authors' contributions } { \ }

All authors made equal contribution.
  
\subsection{Acknowledgments } { \ }

  None.

\end{document}